\documentclass[a4paper, 11pt, reqno]{amsart}

\usepackage{amssymb}
\usepackage[margin=2.5cm]{geometry}
\usepackage{}

\newtheorem{theorem}{Theorem}[section]
\newtheorem{lemma}[theorem]{Lemma}
\newtheorem{lem}[theorem]{Lemma}
\newtheorem{conj}[theorem]{Conjecture}
\newtheorem{prop}[theorem]{Proposition}
\newtheorem{cor}[theorem]{Corollary}
\newtheorem{claim}[theorem]{Claim}

\theoremstyle{definition}

\theoremstyle{remark}

\numberwithin{equation}{section}

\begin{document}

\title{Lagrangians of hypergraphs: The Frankl-F\"uredi conjecture holds almost everywhere}


\author[Mykhaylo Tyomkyn]{Mykhaylo Tyomkyn}
\address{School of Mathematics, Tel Aviv University, Tel Aviv 69978, Israel}
\curraddr{}
\email{tyomkynm@post.tau.ac.il}
\thanks{Supported in part by ERC Starting Grant 633509}


\date{\today}

\begin{abstract}
Frankl and F\"uredi conjectured in 1989 that the maximum Lagrangian of all $r$-uniform hypergraphs of fixed size $m$ is realised by the initial segment of the colexicographic order. In particular, in the principal case $m=\binom{t}{r}$ their conjecture states that every $H\subseteq \mathbb{N}^{(r)}$ of size $\binom{t}{r}$ satisfies
\begin{align*}
\max \{\sum_{A \in H}\prod_{i\in A} y_i \ \colon \ y_1,y_2,\ldots \geq 0; \sum_{i\in \mathbb{N}} y_i=1 \}&\leq \frac{1}{t^r}\binom{t}{r}.
\end{align*}

We prove the above statement for all $r\geq 4$ and large values of $t$ (the case $r=3$ was settled by Talbot in 2002). More generally, we show for any $r\geq 4$ that the Frankl-F\"uredi conjecture holds whenever $\binom{t-1}{r} \leq m \leq \binom{t}{r}- \gamma_r t^{r-2}$ for a constant $\gamma_r>0$, thereby verifying it for `most' $m\in \mathbb{N}$.

Furthermore, for $r=3$ we make an improvement on the results of Talbot~\cite{Tb} and Tang, Peng, Zhang and Zhao~\cite{TPZZ}.
\end{abstract}

\maketitle

\section{Introduction}

Multilinear polynomials are of central interest in most branches of modern mathematics, and extremal combinatorics is by no means an exception. In particular, a large number of hypergraph Tur\'{a}n problems reduce to calculating or estimating the Lagrangian of a hypergraph, which is a constrained maximum of the multilinear function naturally associated with the hypergraph.

To set the scene, we need a few definitions. We follow standard notation of extremal combinatorics (see e.g.~\cite{Bwhite}). In particular, for $n,r\in \mathbb{N}$, we write $[n]$ for the set $\{1,\dots,n\}$  and, given a set $X$, by $X^{(r)}$ we denote the set family $\{A\subseteq X: |A|=r\}$. 
Dealing with finite families of finite sets we will be freely switching between the set system and the hypergraph points of view: with no loss of generality, we can assume our hypergraphs to be defined on $\mathbb{N}$, yet we write $e(H)$ for the number of sets (`edges') in $H$.

For a finite $r$-uniform hypergraph $H\subseteq [n]^{(r)}$ and a vector of real numbers (referred as a \emph{weighting}) $\vec{y}:=(y_1,\dots, y_n)$ consider a multilinear polynomial function 
$$
L(H,\vec{y}):=\sum_{A \in H}\prod_{i\in A} y_i.
$$ 
The \emph{Lagrangian} of $H$ is defined as its maximum on the standard simplex
$$\lambda(H):=\max\{L(H,\vec{y})\colon y_1,\dots, y_n \geq 0; \sum_{i=1}^n y_i=1 \};$$ 
note that, by compactness, the maximum does always exist (but need not be unique).

The above notion was introduced in 1965 by Motzkin and Strauss~\cite{MzSt} for $r=2$, that is for graphs, in order to give a new proof of Tur\'{a}n's theorem. Later it was extended to uniform hypergraphs, where the Lagrangian plays an important role in governing densities of blow-ups. In particular, using Lagrangians of $r$-graphs, Frankl and R\"odl~\cite{FR} disproved a conjecture of Erd\H{o}s~\cite{Erd} by exhibiting infinitely many non-jumps for hypergraph Tur\'{a}n densities. In the following years the Lagrangian has found numerous applications in hypergraph Tur\'{a}n problems; for more details we refer to a survey by Keevash~\cite{Ke} and the references therein. Further results, which appeared after the publication of~\cite{Ke}, include~\cite{HK} and~\cite{TPZZ}.

In this paper we address the problem of maximising the Lagrangian itself over all $r$-graphs  with a fixed number of edges. Let $H^{m,r}$ be the subgraph of $\mathbb{N}^{(r)}$ consisting of the first $m$ sets in the colexicographic order (recall that this is the ordering on $\mathbb{N}^{(r)}$ in which $A<B$ if $\max(A\triangle B)\in B$).
In 1989 Frankl and F\"uredi~\cite{FF} conjectured that the maximum Lagrangian of an $r$-graph on $m$ edges is realised by $H^{m,r}$. 
\begin{conj}[\cite{FF}]\label{conj:FF}
$\lambda(H^{m,r})=\max \{\lambda(H)\colon H\subseteq \mathbb{N}^{(r)}, e(H)=m\}$. 
\end{conj}
In an important special case, which we refer to as the \emph{principal case}, Conjecture~\ref{conj:FF} states that for $m=\binom {t}{r}$ the maximum Lagrangian is attained on $H^{m,r}=[t]^{(r)}$, where we have $\lambda(H^{m,r})=\lambda([t]^{(r)})= \frac{1}{t^r}\binom{t}{r}$. While initially the Frankl-F\"uredi conjecture was motivated by applications to hypergraph Tur\'{a}n problems, we think it also interesting in its own right, as it makes a natural and general statement about maxima of multilinear functions. 

For $r=2$ the validity of Conjecture~\ref{conj:FF} is easy to see and follows from the arguments of Motzkin and Strauss~\cite{MzSt}. In fact, the Lagrangian of a graph $H$ is attained by equi-distributing the weights between the vertices of the largest clique of $H$, resulting in $\lambda(H)=\frac{\omega(H)-1}{2\omega(H)}$. Since $H^{m,r}$ has the largest clique size over all graphs on $m$ edges,  Conjecture~\ref{conj:FF} holds.

On the other hand, the situation for hypergraphs is far more complex, since for $r\geq 3$, unlike in the graph case, no direct way of inferring $\lambda(H)$ from the structure of $H$ is known. Hence one is confined to estimating the Lagrangians of different $r$-graphs against each other without calculating them directly.

For $r=3$ Talbot~\cite{Tb} proved that Conjecture~\ref{conj:FF} holds whenever $\binom{t-1}{3}\leq m \leq \binom{t-1}{3}+\binom{t-2}{2}-(t-1)=\binom{t}{3}-(2t-3)$ for some $t\in \mathbb{N}$. Note that this range covers an asymptotic density $1$ subset of $\mathbb{N}$, and also includes the principal case $m=\binom{t-1}{3}$.
Recently Tang, Peng, Zhang and Zhao~\cite{TPZZ} extended the above range to $\binom{t-1}{3}\leq m \leq \binom{t-1}{3}+\binom{t-2}{2}-\frac{1}{2}(t-1)$. Furthermore, Conjecture~\ref{conj:FF} is known to hold when $\binom{t}{3}-m$ is a small constant, but for the remaining values of $m$ it is still open. 
 
In contrast to this, for $r\geq 4$ much less has been known so far, as Talbot's proof method for $r=3$, perhaps surprisingly, does not immediately transfer. Talbot showed in the same paper~\cite{Tb} that for every $r\geq 4$ there is a constant $\gamma_r>0$ such that if~$\binom{t-1}{r}\leq m\leq\binom{t}{r}-\gamma_rt^{r-2}$ and $H$ is supported on $t$ vertices (that is, ignoring isolated vertices, $H$ is a subgraph of $[t]^{(r)}$), then indeed $\lambda(H)\leq \lambda(H^{m,r})$. 
Still, for no value of $m$, apart from some trivial ones, Conjecture~\ref{conj:FF} has been known to hold. Our main goal in this article is to close this gap by confirming the Frankl-F\"{u}redi Conjecture for `most' values of $m$ for any given $r\geq 4$, including the principal case for large $m$. 
\begin{theorem}\label{thm:main}
For every $r\geq 4$ there exists $\gamma_r>0$ 
such that for all $\binom{t-1}{r} \leq m \leq \binom{t}{r}- \gamma_r t^{r-2}$ we have
$$\lambda(H^{m,r})=\max \{\lambda(H)\colon H\subseteq \mathbb{N}^{(r)}, e(H)=m\}.$$
\end{theorem}
\noindent
\begin{cor}\label{cor:mainclique}
For every $r\geq 4$ there exists $t_r\in \mathbb{N}$ such that for all $t\in \mathbb{N}$ with $t\geq t_r$ we have
$$\max\left\{\lambda(H)\colon H\subseteq \mathbb{N}^{(r)}, e(H)=\binom{t}{r}\right\}=\lambda([t]^{(r)})=\frac{1}{t^r}\binom{t}{r}.$$
\end{cor}
By monotonicity, we obtain another immediate corollary, which can be viewed as a strong approximate version of Conjecture~\ref{conj:FF}.
\begin{cor}\label{cor:FFapprox}
For every $r\geq 4$ there exists $t_r\in \mathbb{N}$ such that for all $t\geq t_r$ the following holds. Suppose that $\binom{t-1}{r}<m\leq \binom{t}{r}$ and that $H$ is an $r$-graph with $e(H)=m$. Then
$$\lambda(H)\leq \frac{1}{t^r}\binom{t}{r}.
$$
\end{cor}
When $H$ is supported on $[t]$ we give a proof of a stronger statement, namely that in this case we can take $\gamma_r=(1+o(1))/(r-2)!$ in Theorem~\ref{thm:main}. More precisely, we claim the following. 
\begin{theorem}\label{thm:refine}
For every $r\geq 3$ there exists a constant $\delta_r>0$ such that for all $\binom{t-1}{r} \leq m \leq \binom{t}{r}- \binom{t-2}{r-2}-\delta_rt^{r-9/4}$ we have 
$$\lambda(H^{m,r})=\max \{\lambda(H)\colon H\subseteq [t]^{(r)}, e(H)=m\}.$$
\end{theorem}

For $r=3$ it was implicitly shown by Talbot in~\cite{Tb} that for any $\binom{t-1}{3}<m\leq \binom{t}{3}$, that is for all $m\in \mathbb{N}$, the $3$-graph maximising the Lagrangian amongst all $m$-edge $3$-graphs can be assumed to be supported on $[t]$. Combined with Theorem~\ref{thm:refine}, this yields, for large $m$, an improvement of the bounds in~\cite{Tb} and~\cite{TPZZ}.
\begin{cor}
There exists a constant $\delta_3>0$ such that for all $\binom{t-1}{3} \leq m \leq \binom{t}{3}- (t-2) -\delta_3t^{3/4}$ we have 
$$\lambda(H^{m,r})=\max \{\lambda(H)\colon H\subseteq \mathbb{N}^{(r)}, e(H)=m\}.$$
\end{cor}

Our proofs use a number of previously known properties of the Lagrangian, as well as induction on $r$ and some facts about uniform set systems such as the Kruskal-Katona theorem. In the following section we shall collect these tools and give an outline of the proofs that will be presented in the subsequent sections.

\section{Notation and preliminaries}\label{sec:prelim}

Let $r\geq 2$ be an integer. Given an $r$-graph $H$ and a set $S\subseteq \mathbb{N}$ with $|S|< r$, the $(r-|S|)$-uniform \emph{link hypergraph} of $S$ is defined as 
$$H_S:=\{A\in \mathbb{N}^{(r-|S|)}\colon A\cup S \in H\}.$$ 
To simplify notation we omit the parentheses and write, for instance, $H_{1,2}$ for $H_{\{1,2\}}$. 

Let us now recall the following standard fact about left-compressed set systems. For the definition of left-compressions (also known as left-shifts) see e.g.~\cite{Bwhite}.

\begin{prop}\label{prop:KK}
Let $n\in \mathbb{N}$ and let $H\subseteq [n]^{(r)}$ be left-compressed with $e(H)=\binom{x}{r}$ for some $n\geq x\geq r$ ($x$ being not necessarily an integer). Then 
\begin{itemize}
\item[(i)]$$e(H_1)\geq \binom{x-1}{r-1}=\frac{r}{x}e(H)$$
\item[(ii)] For all $1\leq j\leq n$ we have $$e(H_j) \leq \frac{r}{j}e(H).$$ 
\end{itemize}
\end{prop}
\begin{proof}[Proof sketch]
(i) is essentially equivalent to Kruskal-Katona theorem and is proved alongside the same lines, see e.g.~\cite{Bwhite}. (ii) follows by double-counting, using that $H$ is left-compressed:
$$je(H_j)\leq \sum_{i=1}^j e(H_i)\leq \sum_{i=1}^n e(H_i) = re(H).
$$ 
\end{proof}
\noindent
Next, we state a well-known inequality for elementary symmetric polynomials, which is a special case of Maclaurin's inequality.
\begin{lem}\label{lem:mcl}
For all $n\in\mathbb{N}$ and $y_1,\dots,y_n\geq 0$ with $\sum_{i\in [n]}y_i=Y$ one has 
\begin{align}\label{eq:mcl}
\sum_{I\in [n]^{(r)}}\prod_{i\in I}y_i \leq \binom{n}{r}\left(\frac{Y}{n}\right)^r<\frac{Y^r}{r!}.
\end{align}
In particular,
\begin{align}\label{eq:mcl2}
\lambda([n]^{(r)})=\frac{1}{n^r}\binom{n}{r}<\frac{1}{r!}.
\end{align}
\end{lem}
Viewing $r$ as constant and $n$ as tending to infinity, we shall make frequent use of the following asymptotics. 
\begin{align}\label{eq:asymptt-1}
\frac{1}{n^r}\binom{n}{r}&=\frac{1}{r!}(1-\binom{r}{2}n^{-1})+O(n^{-2}); \nonumber \\
\frac{1}{(n-1)^r}\binom{n-1}{r}&=\frac{1}{r!}(1-\binom{r}{2}(n-1)^{-1})+O(n^{-2})=\frac{1}{r!}(1-\binom{r}{2}n^{-1})+O(n^{-2}).
\end{align}

A hypergraph $H$ is said to \emph{cover} a vertex pair $\{i,j\}$ if there exists an edge $A\in H$ with $\{i,j\}\subseteq A$.
$H$ is said to \emph{cover pairs} if it covers every pair $\{i,j\}\subseteq \bigcup_{A\in H}A$. Let $H-i$ be the $r$-graph obtain from $H$ by deleting vertex $i$ and the corresponding edges. 

\begin{prop}[\cite{FF}]\label{prop:maxcolex}
Let $n\in \mathbb{N}$ and $H\subseteq [n]^{(r)}$.
\begin{itemize}
\item[(i)] Suppose that $H$ does not cover the pair $\{i,j\}$. Then $\lambda(H)\leq \max\{\lambda(H - i),\lambda(H - j)\}$.
In particular, $\lambda(H)\leq \lambda([n-1]^{(r)})$.
\item[(ii)] Suppose that $m,t\in \mathbb{N}$ satisfy $\binom{t-1}{r}\leq m \leq \binom{t}{r}-\binom{t-2}{r-2}$. Then 
\begin{align}\label{eq:lagrcolex}
\lambda(H^{m,r})=\lambda([t-1]^{(r)})=\frac{1}{(t-1)^r}\binom{t-1}{r}.
\end{align}
\end{itemize} 
\end{prop}

\begin{proof}[Proof sketch]
(i) is obtained by considering all weights except $y_i$ and $y_j$ as fixed. This effectively turns $L(H,\vec{y})$ into a linear function in one variable, which is, ostensibly, maximised at an endpoint of its domain interval. (ii) follows from (i) since on the one hand $[t-1]^{(r)}\subseteq H^{m,r}$, and on the other hand for the above values of $m$ the graph $H^{m,r}$ does not cover the pair $\{t-1,t\}$.
\end{proof}

A simple scaling gives the following fact about the Lagrangians of link hypergraphs.
\begin{lem} 
Suppose $H\subseteq [n]^{(r)}$ and $\vec{y}=(y_1,\dots,y_n)$ is a weighting with $y_j\geq 0$ for all $j\in [n]$ and with $\sum_{j=1}^n y_j=1$. Then for all $i\in[n]$ we have
\begin{align}\label{eq:scaling}
L(H_i,\vec{y})\leq (1-y_i)^{r-1}\lambda(H_i).
\end{align}
\end{lem}
\begin{proof}
If $y_i=1$, then this is self-evident. Otherwise, define the new weighting $z_j=(1-y_i)^{-1}y_j$, for all $j\neq i$ and $z_i=0$. We have $\sum_{j=1}^n z_j=1$ and thus 
$$\lambda(H_i)\geq L(H_i,\vec{z})=(1-y_i)^{-(r-1)}L(H_i,\vec{y}).$$
\end{proof}
Define $H_{i\setminus j}\colon = \{A\in H_i\setminus H_j:j\notin A\}$. In other words, $H_{i\setminus j}$ contains precisely all $r-1$-sets $A$ such that $A\cup \{i\}\in H$ but $A\cup \{j\}\notin H$. 
Note that having $H_{i\setminus j}=H_{j\setminus i}=\emptyset$ is equivalent to $H$ having an automorphism via interchanging $i$ and $j$. This implies the following straightforward fact. 
\begin{prop}\label{prop:symmetry}
Suppose that $H_{i\setminus j}=H_{j\setminus i}=\emptyset$. Then $L(H,\vec{y})\leq L(H,\vec{z})$, where $z_i=z_j=(y_i+y_j)/2$ and $z_\ell=y_\ell$ otherwise.
\end{prop}
\begin{proof}
$$L(H,\vec{z})- L(H,\vec{y})=L(H_{i,j},\vec{y})(z_iz_j-y_iy_j)\geq 0.
$$
\end{proof}
From now on let $r\geq 4$ and suppose that $\binom{t-1}{r}\leq m \leq \binom{t}{r}-\binom{t-2}{r-2}$ for some $t\in \mathbb{N}$. Let $G$ be a graph with $e(G)=m$ which satisfies $\lambda(G)=\max \{\lambda(H)\colon H\subseteq \mathbb{N}^{(r)}, e(H)=m\}$ and let $\vec{x}$ be a weighting attaining the Lagrangian of $G$, that is $x_i\geq 0$ for all $i$, $\sum x_i=1$ and $L(G,\vec{x})=\lambda(G)$ (note that in general $G$ and $\vec{x}$ are not unique).
Following the conventional notation (see e.g.~\cite{Tb}), we can assume by symmetry that the entries of $\vec{x}$ are listed in descending order, that is $x_i\geq x_j$ for all $i<j$. We shall furthermore assume that, subject to the above conditions, $\vec{x}$ has the minimum possible number of non-zero entries, and let $T$ be this number. By the above, we have $x_1\geq\dots \geq x_T>0$ and $x_1+\dots+x_T=1$.

Suppose that $G$ achieves a strictly larger Lagrangian than $H^{m,r}$. By~\eqref{eq:lagrcolex} we have 
\begin{align}\label{eq:basicassumption}
\lambda(G)=L(G,\vec{x}) >\frac{1}{(t-1)^r}\binom{t-1}{r},
\end{align}
which in turn implies $T\geq t$ (otherwise $\lambda(G)= \lambda(G\cap [t-1]^{(r)})\leq \lambda([t-1]^{(r)})$, a contradiction). Our goal is to show that under these assumptions we must have $m> \binom{t}{r}-\gamma_rt^{r-2}$ for some constant $\gamma_r>0$. Before we proceed, let us recall some well-known facts about the newly defined $r$-graph $G$ and its Lagrangian.
\begin{prop}\label{prop:commonsense}
$\bigcup_{A\in G} A = [T]$.
\end{prop}
\begin{proof}
If there exists some $i\in [T]\setminus \bigcup_{A\in G}A$, i.e.~if there is a positive weight not used by any edge, then by re-distributing the weight $x_i$ evenly between the vertices of some edge $A\in G$, we obtain a new weighting $\vec{z}$ with $\sum z_i=1$ and $L(G,\vec{z})>L(G,\vec{x})=\lambda(G)$, a contradiction. Hence, $[T]\subseteq \bigcup_{A\in G}A$.

Conversely, suppose that there is some $A\in G$ with $A\not \subseteq [T]$, i.e.~that an edge of $G$ uses a zero weight. Then, since $T\geq t$ and $m<\binom{t}{r}\leq \binom{T}{r}$, there must be a set $B\in [T]^{(r)}\setminus G$. Now, the $r$-graph $G'=(G\setminus A)\cup B$ has $m$ edges and satisfies $\lambda(G')\geq L(G',\vec{x})>L(G,\vec{x})=\lambda(G)$, contradicting the assumption on $G$. Hence, $\bigcup_{A\in G}A\subseteq [T]$.  
\end{proof}
\begin{prop}[\cite{FF}]\label{prop:lagrbasicfact}
$G$ can be assumed to be left-compressed and to cover pairs.
\end{prop}
\begin{proof}[Proof sketch]
If $G'$ is obtained from $G$ by a series of left-shifts, it is clear that $$\lambda(G)=L(G,\vec{x})\leq L(G',\vec{x})\leq \lambda(G').$$ Additionally, left-compressions cannot increase $T$. 
The fact that $G$ covers pairs follows by  Proposition~\ref{prop:maxcolex}(i) and the definition of $T$.
\end{proof}
\noindent
The next statement holds in more generality, but we shall mainly need it for $G$ and $\vec{x}$. 
\begin{prop}[\cite{FR},\cite{Tb}]\label{prop:derivatives}
 Let $G$, $T$ and $\vec{x}$ be as defined above. Then
\begin{itemize}
\item[(i)] For all $1\leq i \leq T$ 
one has
\begin{align}\label{eq:lagridentity}
L(G_i,\vec{x}) = r\lambda(G).
\end{align}
\item[(ii)] For all $1\leq i<j \leq T$ 
one has
\begin{align}\label{eq:lagridentity2}
(x_i-x_j)L(G_{i,j},\vec{x})=L(G_{i\setminus j},\vec{x}).
\end{align}
\end{itemize}
\end{prop} 
\begin{proof}[Proof sketch]
Notice that $L(G_i,\vec{x})$ is the partial derivative of $L(G,\vec{x})$ with respect to $x_i$. By considering Lagrange multipliers, one obtains $L(G_i,\vec{x})=L(G_j,\vec{x})$ for all $i\neq j$. Hence, Proposition~\ref{prop:commonsense} yields
$$L(G,\vec{x})=\frac{1}{r}\sum_{j=1}^T L(G_j,\vec{x})x_i=\frac{1}{r}L(G_i,\vec{x})\sum_{j=1}^T x_j=\frac{1}{r}L(G_i,\vec{x}),$$
proving (1). To see that (2) also holds, note again that $L(G_i,\vec{x})=L(G_j,\vec{x})$ and use the fact that $G$ is left-compressed. 
\end{proof}

In order to prove Theorem~\ref{thm:main} we apply induction on $r$, assuming Corollary~\ref{cor:FFapprox} for $r-1$ as the induction hypothesis. Since Theorem~\ref{thm:main} is concerned with large values of $m$, we will assume $t$ to be greater than any given number (which may also depend on $r$), whenever we need it. We do not attempt to optimise $\gamma_r$.

As our induction base we take Corollary~\ref{cor:FFapprox} for $r=3$, which is known to hold by Talbot's theorem~\cite{Tb}. Note though, that our proof does not crucially rely on~\cite{Tb}, as, alternatively, we could start at $r=3$, taking the trivial $r=2$ case as the induction base; this would also give a new proof of a slightly weaker form of Talbot's theorem for $r=3$. 

The rest of the paper is organised as follows. In Section~\ref{sec:coarse} we establish first upper bounds on $T$ and $x_1$. With this information we show in Section~\ref{sec:tails} that if $T$ is greater than $t$ by some additive term, then $x_T$ is less than $1/t$ by a similar multiplicative term. This implies that, for a certain index $q<t$, $x_q$ is significantly smaller than $1/t$. This fact will, in turn, be applied in Section~\ref{sec:x1}, where we prove that $T-t\leq C_0$ for a constant $C_0(r)$, which allows in Section~\ref{sec:smallsup}, using a refinement of an argument from~\cite{TPZZ}, to bound $\binom{t}{r}-m$, concluding the proof of Theorem~\ref{thm:main}. In Section~\ref{sec:refinement} we apply a version of the argument in Section~\ref{sec:smallsup} in order to prove Theorem~\ref{thm:refine}. In the final section we discuss possible ways of extending our results.

\section{Some coarse bounds}\label{sec:coarse}
In this section we prove some, rather crude, first bounds on $T$ and $x_1$. They will be required for establishing tighter bounds later on. 
\begin{lem}\label{lem:10t}
$T<10t$.
\end{lem}
\begin{proof}
We have
\begin{align*}
\frac{1}{(t-1)^r}\binom{t-1}{r}\stackrel{\eqref{eq:basicassumption}}{<} L(G,\vec{x})\stackrel{\eqref{eq:lagridentity}}{=}\frac{1}{r}L(G_T,\vec{x})\stackrel{\eqref{eq:scaling}}{\leq} \frac{1}{r}(1-x_T)^{r-1}\lambda(G_T)\leq \frac{\lambda(G_T)}{r}.
\end{align*}
So,
\begin{align}\label{eq:sec3asympt}
\lambda(G_T)\geq \frac{r}{(t-1)^r}\binom{t-1}{r}.
\end{align}
Let $s\in \mathbb{N}$ be such that
\begin{align}\label{eq:eGT}
\binom{s-1}{r-1}< e(G_T) \leq \binom{s}{r-1}.
\end{align}
Since, by Proposition~\ref{prop:lagrbasicfact}, $G$ is left-compressed, Proposition~\ref{prop:KK}(ii) implies 
\begin{align}\label{eq:GTbasic}
e(G_T)\leq\frac{r}{T}m<\frac{r}{T}\binom{t}{r}\leq \frac{r}{t}\binom{t}{r}=\binom{t-1}{r-1},
\end{align}
which means $s<t$.
We claim that $s$ cannot be too small either, or, in other words, that $e(G_T)$ is `reasonably large'. 
\begin{claim}\label{cl:sislarge}
$$s\geq (1+o(1))\frac{r-2}{r}t.$$
\end{claim}
\begin{proof}
By the induction hypothesis and monotonicity, we have
$$\frac{r}{(t-1)^r}\binom{t-1}{r}\stackrel{~\eqref{eq:sec3asympt}}{\leq} \lambda(G_T)\leq \frac{1}{{s_1}^{r-1}}\binom{{s_1}}{r-1},
$$ 
where ${s_1}=\max\{s,t_{r-1}\}$ and $t_{r-1}$ is as in Corollary~\ref{cor:FFapprox}. Since we can assume that $t>10t_{r-1}$, so that $t_{r-1}<t/10$, it follows that
\begin{align}\label{eq:s2}
\frac{r}{(t-1)^r}\binom{t-1}{r}\leq \frac{1}{s_2^{r-1}}\binom{s_2}{r-1},
\end{align}
where 
$s_2=\max\{s,t/10\}$. 
Suppose now that $s<t/10$. Then $s_2=t/10$ and we get
\begin{align}\label{eq:srange}
\frac{1}{(r-1)!}\cdot (1-\binom{r}{2}\frac{1}{t})+O({t^{-2}})&\stackrel{\eqref{eq:asymptt-1}}{=}\frac{r}{(t-1)^r}\binom{t-1}{r}\stackrel{\eqref{eq:s2}}{\leq} \frac{1}{(t/10)^{r-1}}\binom{t/10}{r-1}\nonumber \\
&\stackrel{\eqref{eq:asymptt-1}}{=}\frac{1}{(r-1)!}\cdot (1-\binom{r-1}{2}\frac{10}{t})+O({t^{-2}}),
\end{align}
which results in a contradiction, as $10\binom{r-1}{2}>\binom{r}{2}$.
Thus we must have $s\geq t/10$, so that $s_2=s$ and
$$\frac{r}{(t-1)^r}\binom{t-1}{r}\leq \frac{1}{{s}^{r-1}}\binom{{s}}{r-1}.
$$ 
Applying~\eqref{eq:asymptt-1} and the just established fact that $t/10 \leq s \leq t$, i.e. $s=\Theta(t)$, similarly to~\eqref{eq:srange} we obtain
$$1-\binom{r-1}{2}\frac{1}{s}\geq 1-\binom{r}{2}\frac{1}{t}+O\left({t^{-2}}\right).
$$
Equivalently,
$$\binom{r}{2}s\geq \binom{r-1}{2}t+O(1)=(1+o(1)) \binom{r-1}{2}t.
$$
Thus,
$${s}\geq (1+o(1))\frac{\binom{r-1}{2}}{\binom{r}{2}}t=(1+o(1))\frac{r-2}{r}t.$$ 
\end{proof}
\noindent
Notice now that
$$\frac{r}{T}\binom{t}{r}\stackrel{\eqref{eq:GTbasic}}{\geq} e(G_T)\stackrel{\eqref{eq:eGT}}{\geq} \binom{s-1}{r-1} =\frac{r}{s}\binom{s}{r}.
$$
Thus, using Claim~\ref{cl:sislarge}, we obtain
\begin{align*}
T\leq \frac{s\binom{t}{r}}{\binom{s}{r}}&= (1+o(1)) s\cdot \left(\frac{t}{s}\right)^{r}= (1+o(1)) t\cdot \left(\frac{t}{s}\right)^{r-1}\leq(1+o(1))t\left(\frac{r}{r-2}\right)^{r-1} \\
&=(1+o(1))t\cdot \left(1-\frac{2}{r}\right)^{-r+1} <10t,
\end{align*}
where the last inequality is due to the fact that $e^2<8$, with the latter replaced by $10$ to account for small values of $r$. This completes the proof of Lemma~\ref{lem:10t}.
\end{proof}
\noindent
Next, we give an upper bound on $x_1$.
\begin{lem}\label{lem:coarsex1} 
$x_1\leq {r}/{t}$.
\end{lem}
\begin{proof}
Observe that 
\begin{align}\label{eq:lagrGT}
L(G,\vec{x})\stackrel{\eqref{eq:lagridentity}}{=}\frac{1}{r}L(G_1,\vec{x}) \stackrel{\eqref{eq:scaling}}{\leq} \frac{1}{r}(1-x_1)^{r-1}\lambda(G_1)\stackrel{\eqref{eq:mcl2}}{<} \frac{(1-x_1)^{r-1}}{r!}.  
\end{align}
Hence, by~\eqref{eq:basicassumption}, we must have  
$$
\frac{1}{(t-1)^r}\binom{t-1}{r}< \frac{(1-x_1)^{r-1}}{r!},
$$
which, due to~\eqref{eq:asymptt-1}, means
$$(1-x_1)^{r-1}> 1-\binom{r}{2}\frac{1}{t} +O(t^{-2}).  
$$
Since $(1-x_1)^{r-1}$ is a decreasing function of $x_1$ and $(1-r/t)^{r-1}=1-2\binom{r}{2}/t +O(t^{-2})<1-\binom{r}{2}/t +O(t^{-2})$, we must have $x_1\leq {r}/{t}$. 
\end{proof}

\section{Bounding the tails}\label{sec:tails}

With the above information we aim to establish some upper bounds on $x_i$ for large values of $i$. 
Put $\Delta:=T-t$, so that by Lemma~\ref{lem:10t} we have $\Delta\leq 9t$. For technical reasons we will assume here that $\Delta\geq 1$; the case $\Delta=0$ will be dealt with in Section~\ref{sec:smallsup}. 
\begin{lem}\label{lem:Tbound}
If $\Delta\geq 1$ then 
\begin{align}\label{eq:xTbound}
x_T\leq \frac{10}{\Delta^{1/(r-1)} t}.
\end{align}
\end{lem}
\begin{proof}
Recall from Proposition~\ref{prop:derivatives} that 
$$(x_1-x_{T})L(G_{1,T},\vec{x})=L(G_{1\setminus T},\vec{x}),
$$
and note that, by Lemma~\ref{lem:mcl}, we have $L(G_{1,T},\vec{x})\leq 1/(r-2)!$. Combining these facts with Lemma~\ref{lem:coarsex1}, we obtain
$$\frac{r}{t(r-2)!}\geq\frac{x_1}{(r-2)!}\geq (x_1-x_T)L(G_{1,T},\vec{x})=L(G_{1\setminus T},\vec{x})\geq e(G_{1\setminus T})x_T^{r-1}.
$$
Thus,
\begin{align}\label{eq:xTeG1/T}
x_T\leq \left(\frac{r}{t(r-2)!\cdot e(G_{1/T})}\right)^{{1}/(r-1)}.
\end{align}
Now, observe that, since $G$ is left-compressed, we have 
\begin{align}\label{eq:linkrgraphs}
e(G_{1\setminus T})&=e(G_1)-e(G_{1,T})-e(G_1\cap G_T)=e(G_1)-e(G_{1,T})-(e(G_T)-e(G_{1,T}))\nonumber \\
&=e(G_1)-e(G_T).
\end{align} 
Let $x$ be the real number satisfying $m=\binom{x}{r}$, so that $t-1\leq x<t$. 
By Proposition~\ref{prop:KK}(i) we have 
$$e(G_1)\geq \frac{rm}{x}>\frac{rm}{t}.$$ On the other hand, by Proposition~\ref{prop:KK}(ii), we have
$e(G_T)\leq {rm}/{T}$. Together with Lemma~\ref{lem:10t} this implies
\begin{align}\label{eq:eG1T}
e(G_{1\setminus T})&\stackrel{\eqref{eq:linkrgraphs}}{=}e(G_1)-e(G_T)\geq (\frac{1}{t}-\frac{1}{T})rm=\frac{\Delta r m}{tT}\geq \frac{\Delta r m}{10t^2}
>\frac{\Delta r}{10t^2}\cdot \frac{3}{4}\binom{t}{r}\nonumber\\
&>\frac{\Delta t^{r-2}}{20(r-1)!}.
\end{align}
Hence, combining~\eqref{eq:xTeG1/T} and~\eqref{eq:eG1T}, we obtain
$$x_T\leq \left(\frac{r}{t(r-2)!}\cdot \frac{20(r-1)!}{\Delta t^{r-2}}\right)^{1/(r-1)}=\left(\frac{20r(r-1)}{\Delta t^{r-1}}\right)^{1/(r-1)}< \frac{10}{\Delta^{1/(r-1)}t}.
$$
\end{proof}
\noindent
Next, we want to establish a similar upper bound for $x_q$ where $q$ is `somewhat smaller' than $t$; this will prove crucial in due course. More precisely, let $q\in \mathbb{N}$ be such that
\begin{align}\label{eq:defq}
\binom{q-1}{r-1}\leq\frac{t}{T-1}\binom{t-1}{r-1}< \binom{q}{r-1}.
\end{align}
The following technical lemma bounds $q$ from above and below.

\begin{lem}\label{lem:qestimate}
We have
\begin{itemize}
\item[(i)] 
$$t\cdot \left(\frac{t}{T}\right)^{\frac{1}{r-1}}-r<q<t\cdot \left(\frac{t}{T}\right)^{\frac{1}{r-1}}+r.
$$ 
\noindent
In particular, 
\begin{align}\label{eq:qThetat}
q=(1+o(1))t\left(\frac{t}{T}\right)^{\frac{1}{r-1}}=\Theta(t).
\end{align}
\item[(ii)] If $\Delta>4r$ then
$ t-q<\Delta.
$
\end{itemize}
\end{lem}
\begin{proof}
For the upper bound in (i) observe that
$$\frac{(q-r)^{r-1}}{(r-1)!}<\binom{q-1}{r-1}\stackrel{\eqref{eq:defq}}{\leq}\frac{t}{T-1}\binom{t-1}{r-1}<\frac{t}{T}\frac{t^{r-1}}{(r-1)!},
$$
where the last inequality uses $(t-1)/(T-1)\leq t/T$.
The lower bound in (i) follows similarly:
$$\frac{q^{r-1}}{(r-1)!}>\binom{q}{r-1}\stackrel{\eqref{eq:defq}}{>}\frac{t}{T-1}\binom{t-1}{r-1}>\frac{t}{T}\frac{(t-r)^{r-1}}{(r-1)!}.
$$
With these bounds, $q=\Theta(t)$ follows from $T=\Theta(t)$, which we know to hold by Lemma~\ref{lem:10t}.

To show (ii), notice that this is self-evident when $\Delta\geq t$, and for $\Delta<t$ it suffices to verify the inequality $\binom{t-\Delta}{r-1}<\frac{t}{T-1}\binom{t-1}{r-1}$. Substituting $T=t+\Delta$ yields
$$\frac{t}{T-1}\binom{t-1}{r-1}-\binom{t-\Delta}{r-1}>\frac{t(t-r)^{r-1}-(t+\Delta)(t-\Delta)^{r-1}}{(T-1)(r-1)!}.
$$
Since $(t+\Delta)(t-\Delta)^{r-1}=(t^2-\Delta^2)(t-\Delta)^{r-2}$ is a decreasing function of $\Delta$ in the domain $\Delta\in [0,t]$, for $\Delta>4r$ we get 
\begin{align*}
t(t-r)^{r-1}-(t+\Delta)(t-\Delta)^{r-1}&>t(t-r)^{r-1}-(t+4r)(t-4r)^{r-1}\\
&=(4r(r-2)-r(r-1))t^{r-2}+O(t^{r-3})>0.
\end{align*} 
 
\end{proof}
\noindent
Now let us estimate $x_q$.
\begin{lem}\label{lem:pushbeyond}
If $\Delta\geq 1$ then
\begin{align*}
x_q\leq \frac{10r\Delta^{-1/{(r-1)^2}}}{t}.
\end{align*}
\end{lem}
\begin{proof}
By Proposition~\ref{prop:KK}(ii) we have  
\begin{align}\label{eq:lazy}
e(G_{T-1})\leq \frac{r}{T-1}m\leq \frac{r}{T-1}\binom{t}{r}=\frac{t}{T-1}\binom{t-1}{r-1}
\stackrel{\eqref{eq:defq}}{<}\binom{q}{r-1}.
\end{align}
Since, by Proposition~\ref{prop:lagrbasicfact}, $G$ covers pairs, there exists some $A\in G$ with $\{T-1,T\} \subseteq A$; note that $L(A,\vec{x})\leq x_1^{r-2}x_{T-1}x_T$.
On the other hand, since, also by Proposition~\ref{prop:lagrbasicfact}, $G$ is left-compressed, for the set $B=\{q-r+2,\dots,q\}\cup \{T-1\}$ we must have $B\notin G$, for otherwise we would have $[q]^{(r-1)}\subseteq G_{T-1}$, contradicting~\eqref{eq:lazy} (this is the reasoning behind the definition of $q$).

Note that if we had $L(B,\vec{x})>L(A,\vec{x})$, then for the $r$-graph $G'=(G\setminus \{A\})\cup \{B\}$, with $e(G')=m$, we would have $\lambda(G')\geq L(G',\vec{x})>  L(G,\vec{x})=\lambda(G)$, a contradiction. Therefore,
\begin{align*}
x_1^{r-2}x_{T-1}x_T \geq L(A,\vec{x})\geq L(B,\vec{x})\geq x_q^{r-1}x_{T-1}.
\end{align*} 
Together with Lemma~\ref{lem:coarsex1} and Lemma~\ref{lem:Tbound}, this gives
\begin{align*}
x_q\leq (x_1^{r-2}x_T)^{1/(r-1)}\leq \left(\frac{r^{r-2}}{t^{r-2}}\cdot \frac{10}{\Delta^{1/(r-1)}t}\right)^{1/(r-1)}\leq\frac{10r}{\Delta^{1/(r-1)^2}t}.
\end{align*}
\end{proof}
\noindent
Combining Lemma~\ref{lem:qestimate} and Lemma~\ref{lem:pushbeyond} we obtain the following upper bound on the sum of all `small' weights. 
\begin{cor}\label{cor:boundthetail} If $\Delta>4r$ then
$$\sum_{i={q+1}}^{T}x_i\leq 2\Delta x_q \leq \frac{20r\Delta^{1-\frac{1}{(r-1)^2}}}{t}.
$$
\end{cor}

\section{A better estimate for $T$}\label{sec:x1}

Our next task will be to prove that $\Delta$ is in fact bounded above by some (large) constant.

\begin{lem}\label{lem:bigmain}
There exists a constant $C_0(r)$ such that $\Delta\leq C_0(r)$.
\end{lem}

\begin{proof}

If $\Delta\leq 4r$, then there is nothing to prove. Thus, suppose that $\Delta>4r$ , so that Lemma~\ref{lem:qestimate} and Corollary~\ref{cor:boundthetail} apply. It turns out that in this situation we can generously add all the missing edges to $G$ and show that $L([T]^{(r)},\vec{x})\leq\lambda([t-1]^{(r)})$, unless $\Delta\leq C_0$.
 
More precisely, let $S:=\sum_{i=q+1}^T x_i$, so that by Corollary~\ref{cor:boundthetail} we have $S\leq {20r\Delta^{1-\frac{1}{(r-1)^2}}}/{t}$. Let $\vec{z}$ be the weighting defined by $z_1=\dots=z_q=(1-S)/q$ and $z_{q+1}=\dots=z_T=S/(T-q)$. It is an immediate consequence of Proposition~\ref{prop:symmetry} that $L([T]^{(r)},\vec{z})=\max\{L([T]^{(r)},\vec{y})\colon y_1,\dots,y_T\geq 0, \sum_{i=1}^q y_i =1-S,\sum_{i=q+1}^T y_i =S\}$, so that, in particular, $L([T]^{(r)},\vec{x})\leq L([T]^{(r)},\vec{z})$. Therefore, we have
\begin{align}\label{eq:largeclique}
\lambda(G)=L(G,\vec{x})\leq L([T]^{(r)},\vec{x})\leq L([T]^{(r)},\vec{z}).
\end{align}
Estimating the latter, we obtain
\begin{align}\label{eq:generous}
L([T]^{(r)},\vec{z}) &= \binom{q}{r}\cdot \frac{1}{q^r}(1-S)^r+\sum_{p=1}^r \binom{T-q}{p}\binom{q}{r-p}\frac{S^{p}}{(T-q)^{p}}\frac{(1-S)^{r-p}}{q^{r-p}}\nonumber \\
&\stackrel{\eqref{eq:asymptt-1},\eqref{eq:qThetat}}{\leq} \frac{1}{r!}(1-S)^r(1-{\binom{r}{2}}{q^{-1}}) +O(t^{-2})+\sum_{p=1}^r \frac{1}{p!(r-p)!}S^{p}(1-S)^{r-p}\nonumber \\
&\leq\frac{1}{r!}\left(-(1-S)^r{\binom{r}{2}}{q^{-1}}+\sum_{p=0}^r \binom{r}{p}S^{p}(1-S)^{r-p}\right)+O(t^{-2})\nonumber \\
&=\frac{1}{r!}\left(1-(1-S)^r{\binom{r}{2}}{q^{-1}}\right)+O(t^{-2}).
\end{align}
On the other hand, as before, we have
\begin{align}\label{eq:restatetheobvious}
\lambda(G)\stackrel{\eqref{eq:basicassumption}}{>} \frac{1}{(t-1)^r}\binom{t-1}{r}\stackrel{\eqref{eq:asymptt-1}}{=}\frac{1}{r!}(1-{\binom{r}{2}}{t^{-1}})+O(t^{-2}).
\end{align}
Hence,~\eqref{eq:largeclique},~\eqref{eq:generous} and~\eqref{eq:restatetheobvious} together imply
$$ {-(1-S)^r}/{q}\geq -{1}/{t}+O(t^{-2}),$$
which, due to $q=\Theta(t)$ by~\eqref{eq:qThetat}, yields
\begin{align*}
q\geq t(1-S)^r+O(1)>t(1-rS)+O(1).
\end{align*}
Now, invoking Lemma~\ref{lem:qestimate}(i) and Corollary~\ref{cor:boundthetail} 
we obtain
$$r+t(1+\Delta/t)^{-1/(r-1)}\geq t(1-20r^2 \Delta^{1-1/(r-1)^2}t^{-1})+O(1),
$$
thus
\begin{align}\label{eq:finalmess}
(1+\Delta/t)^{-1/(r-1)}\geq 1-20r^2 \Delta^{1-1/(r-1)^2}t^{-1}+O(t^{-1}).
\end{align}
Now, on the one hand, if $\Delta=\Omega(t)$, then the left hand side of~\eqref{eq:finalmess} is $1-\Omega(1)$ while the right hand side is $1-o(1)$, a contradiction. On the other hand, for $\Delta=o(t)$ we can write 
$$ (1+\Delta/t)^{-1/(r-1)}<1-\frac{\Delta}{2t(r-1)}<1-\frac{\Delta}{2rt}.
$$    
With this,~\eqref{eq:finalmess} implies
$$1-\frac{\Delta}{2rt}\geq 1-\frac{20r^2 \Delta^{1-1/(r-1)^2}}{t}+O(t^{ -1}),
$$
or, equivalently,
$$\Delta\leq 40r^3\Delta^{1-1/(r-1)^2}+O(1).
$$
The last inequality can only hold if $\Delta$ is bounded. Hence,
$\Delta \leq C_0(r)$.
\end{proof}

\section{The small support case}\label{sec:smallsup}

Lemma~\ref{lem:bigmain} entails $T = t+C$, where $C$ is at most a constant: $0\leq C\leq C_0(r)$. In this section we apply this fact to complete the proof of Theorem~\ref{thm:main}.

First, we claim that Lemma~\ref{lem:bigmain} yields a stronger upper bound on $x_1$.
\begin{lem}\label{lem:finex1}
$x_1< 1/(t-\alpha)$ for some constant $\alpha=\alpha(r)\in\mathbb{N}$.
\end{lem}
\begin{proof}
We have
$$
L(G,\vec{x}) \stackrel{\eqref{eq:lagridentity},\eqref{eq:scaling}}{\leq} \frac{1}{r}(1-x_1)^{r-1}\lambda(G_1)\leq \frac{1}{r}(1-x_1)^{r-1}\lambda([T-1]^{(r-1)}) \stackrel{\eqref{eq:mcl2}}{=} \frac{1}{r}\left(\frac{1-x_1}{T-1}\right)^{r-1}\binom{T-1}{r-1}.  
$$
Hence, by~\eqref{eq:basicassumption} we must have  
$$
\frac{1}{(t-1)^r}\binom{t-1}{r}< \frac{1}{r}\left(\frac{1-x_1}{T-1}\right)^{r-1}\binom{T-1}{r-1}.
$$
Applying~\eqref{eq:asymptt-1}, Lemma~\ref{lem:coarsex1} and Lemma~\ref{lem:bigmain}, we obtain
\begin{align*}
1-(r-1)x_1+O(t^{-2})&> \left(1-\binom{r}{2}\frac{1}{t}+O(t^{-2})\right)\left(1-\binom{r-1}{2}\frac{1}{t+C}+O(t^{-2})\right)^{-1}\\
&=\left(1-\binom{r}{2}\frac{1}{t}+O(t^{-2})\right)\left(1-\binom{r-1}{2}\frac{1}{t}+O(t^{-2})\right)^{-1}\\
&=\left(1-\binom{r}{2}\frac{1}{t}+O(t^{-2})\right)\left(1+\binom{r-1}{2}\frac{1}{t}+O(t^{-2})\right)\\
&=1-\binom{r}{2}\frac{1}{t}+\binom{r-1}{2}\frac{1}{t}+O(t^{-2})\\
&=1-(r-1)\frac{1}{t}+O(t^{-2}).
\end{align*}
Thus, 
$$x_1< \frac{1}{t}+O(t^{-2})< \frac{1}{t-\alpha}.  
$$
for some constant $\alpha=\alpha(r)$.
\end{proof}
Without loss of generality, we can assume that $\alpha\geq C$, for otherwise rename $\alpha$ to be $\max\{\alpha,C\}$.

\begin{lemma}\label{lem:trick}
In the above setting, $x_1< 2x_{t-3\alpha}$.
\end{lemma} 

\begin{proof}
Suppose otherwise. Then 
$$2\alpha x_1\geq 4\alpha x_{t-3\alpha}\geq (3\alpha+C)x_{t-3\alpha}\geq \sum_{i=t-3\alpha+1}^T x_i,$$ 
and therefore 
$$(t-\alpha)x_1 = (t-3\alpha)x_1+2\alpha x_1 \geq \sum_{i=1}^{t-3\alpha}x_i+\sum_{i=t-3\alpha+1}^T x_i = \sum_{i=1}^T x_i=1,$$
a contradiction.
\end{proof}
\noindent
As an immediate consequence of Lemma~\ref{lem:trick} we obtain 
\begin{align}\label{eq:2hochr}
x_1^{r-2}x_{t-1}x_t < 2^{r-2}x_{t-3\alpha}^{r-2}x_{t-1}x_t.
\end{align}
\begin{lem}\label{lem:swap}
Let 
\begin{itemize}
\item $\mathcal{A}:=\{A\in G\colon |A\cap \{t-1,\dots,T\}|\geq 2\}$,  
\item $\mathcal{B}:= \{B\in[t]^{(r)}\setminus G\colon |B\cap\{t-1,t\}|\leq 1 \text{ and } | B\cap \{t-3\alpha+1,\dots,t\}| \leq 2\}$. 
\end{itemize}
Then for all $A\in \mathcal{A}$ and $B\in \mathcal{B}$ we have 
$$L(B,\vec{x})> L(A,\vec{x})/2^{r-2}.$$
\end{lem}
\begin{proof}
$$L(A,\vec{x})\leq x_1^{r-2}x_{t-1}x_t\stackrel{\eqref{eq:2hochr}}{<}2^{r-2}x_{t-3\alpha}^{r-2}x_{t-1}x_t\leq 2^{r-2} L(B,\vec{x}).
$$ 
\end{proof}
\noindent
Now we are ready to conclude our main theorem. 
\begin{proof}[Proof of Theorem~\ref{thm:main}]
Let $G$, $T$ and $\vec{x}$ be as defined in Section~\ref{sec:prelim}, in particular,~\eqref{eq:basicassumption} holds. By Lemma~\ref{lem:bigmain} this implies $T = t+C$ where $C<C_0(r)$. In this case, with the notation of Lemma~\ref{lem:swap}, we can assume that $\mathcal{A}\neq \emptyset$, for otherwise $G$ does not cover pairs, contradicting Proposition~\ref{prop:lagrbasicfact}. Suppose now that $|\mathcal{B}|\geq 2^{r-2}|\mathcal{A}|$, and let $G':=(G\setminus \mathcal{A})\cup \mathcal{B}$ (it does not matter that $e(G')>m$). Then Lemma~\ref{lem:swap} entails
$$\lambda(G')\geq L(G',\vec{x})>L(G,\vec{x})=\lambda(G).
$$
However, $G'$ does not cover any pair in $\{t-1,\dots,T\}^{(2)}$, therefore, by applying Proposition~\ref{prop:maxcolex}(i) $C+1$ times, we obtain $\lambda([t-1]^{(r)}) \geq \lambda(G')> \lambda(G)$, a contradiction. Thus, 
$$|\mathcal{B}|< 2^{r-2}|\mathcal{A}|\leq 2^{r-2}\binom{C+2}{2}\cdot \binom{T-2}{r-2}\leq  2^{2r}C^2\binom{t-2}{r-2}.
$$
Hence, we must have
\begin{align*}
\binom{t}{r}-m&= |[t]^{(r)}|-e(G)\leq |[t]^{(r)}\setminus G|\\&\leq |\mathcal{B}|+|\{C\in [t]^{(r)}: \{t-1,t\}\subseteq C\}|+ |\{D\in  [t]^{(r)}: |\{t-3\alpha+1,\dots,t\}\cap D| \geq 3\}|\\
&\leq  (2^{2r}C^2+1)\binom{t-2}{r-2} + O(t^{r-3})
\leq \gamma_r t^{r-2},
\end{align*}
as claimed. Thus the proof of Theorem~\ref{thm:main} is completed.
\end{proof}

\section{A refinement for $T\leq t$}\label{sec:refinement}

Here we prove Theorem~\ref{thm:refine}. In this section let $G$ be a graph on $[t]$ with $e(G)=m$ and $\binom{t-1}{r}\leq m \leq \binom{t}{r}-\binom{t-2}{r-2}$, maximising the Lagrangian amongst all $r$-graphs on $[t]$ with $m$ edges. Suppose that $\lambda(G)>\lambda(H^{m,r})=\frac{1}{(t-1)^r}\binom{t-1}{r}$. Our aim is to show that then $m\geq \binom{t}{r}- \binom{t-2}{r-2}-\delta_rt^{r-9/4}$ for a constant $\delta_r>0$. 

Let $\vec{x}$ be a weighting attaining $\lambda(G)$; we can assume that $\vec{x}$ has exactly $t$ non-zero entries (otherwise $\lambda(G)\leq \lambda([t-1]^{(r)})$, a contradiction) and that the entries of $\vec{x}$ are listed in descending order. It is a straightforward check that Propositions~\ref{prop:lagrbasicfact} and~\ref{prop:derivatives} (with $T=t$) remain valid for the just defined $G$ and $\vec{x}$. 
Hence, analogously to Lemma~\ref{lem:finex1} (with no need to establish Lemma~\ref{lem:bigmain}) we obtain the following upper bound on $x_1$.
\begin{lem}
$x_1<1/(t-r+1)$. 
\end{lem}
\begin{proof}
We have
\begin{align*}
\frac{1}{(t-1)^r}\binom{t-1}{r}&<L(G,\vec{x}) \stackrel{\eqref{eq:lagridentity},\eqref{eq:scaling}}{\leq} \frac{1}{r}(1-x_1)^{r-1}\lambda(G_1)\leq \frac{1}{r}(1-x_1)^{r-1}\lambda([t-1]^{(r-1)})\\ &\stackrel{\eqref{eq:mcl2}}{=} \frac{1}{r}\left(\frac{1-x_1}{t-1}\right)^{r-1}\binom{t-1}{r-1}.  
\end{align*}
This translates to 
$$(1-x_1)^{r-1}>\frac{t-r}{t-1}.
$$
Since $(1-x_1)^{r-1}$ is a decreasing function of $x_1$, it suffices to verify that $(1-1/(t-r+1))^{r-1}< (t-r)/(t-1)$, or, equivalently, that $(t-r)^{r-2}(t-1)<(t-r+1)^{r-1}$. The latter inequality holds by AM-GM.
\end{proof}

Now, with hindsight, select an integer $k\sim t^{1/4}$.
\begin{lem}
\begin{align}\label{eq:crunch}
x_1<\frac{k+1}{k}x_{t-(k+1)r}.
\end{align}
\end{lem}
\begin{proof}
Suppose otherwise. Then 
\begin{align}\label{eq:rkx1}
rk x_1\geq r(k+1)x_{t-(k+1)r}\geq \sum_{i=t-(k+1)r+1}^t x_i,
\end{align}
 and therefore 
$$(t-r+1)x_1 >(t-r(k+1))x_1+rkx_1\stackrel{\eqref{eq:rkx1}}{\geq} \sum_{i=1}^{t-r(k+1)} x_i + \sum_{i=t-r(k+1)+1}^{t} x_i = \sum_{i=1}^t x_i=1,$$
a contradiction. 
\end{proof}
\begin{lem}\label{lem:swap2}
Let 
\begin{itemize}
\item $\mathcal{A}:=\{A\in G\colon\{t-1,t\}\subseteq A\}$ and 
\item $\mathcal{B}:= \{B\in [t]^{(r)}\setminus G\colon |B\cap\{t-1,t\}|\leq 1 \text{ and } | B\cap\{t-(k+1)r+1,\dots,t\}| \leq 2\}$. 
\end{itemize}
Then for all $A\in \mathcal{A}$ and $B\in \mathcal{B}$ we have 
$$L(B,\vec{x})> L(A,\vec{x})\cdot\left(\frac{k}{k+1}\right)^{r-2}.$$
\end{lem}
\begin{proof}
$$L(A,\vec{x})\leq x_1^{r-2}x_{t-1}x_t\stackrel{\eqref{eq:crunch}}{<}\left(\frac{k+1}{k}\right)^{r-2}x_{t-(k+1)r}^{r-2}x_{t-1}x_t\leq \left(\frac{k+1}{k}\right)^{r-2} L(B,\vec{x}).
$$ 
\end{proof} 
\begin{proof}[Proof of Theorem~\ref{thm:refine}]
If $\mathcal{A}=\emptyset$ then $G$ does not cover pairs, contradicting Proposition~\ref{prop:lagrbasicfact}; 
so, we can assume that $\mathcal{A}\neq \emptyset$. Suppose that $|\mathcal{B}|\geq (\frac{k+1}{k})^{r-2}|\mathcal{A}|$, and let $G':=(G\setminus \mathcal{A})\cup \mathcal{B}$ (it does not matter that $e(G')>m$). Then Lemma~\ref{lem:swap2} entails
$$\lambda(G')\geq L(G',\vec{x})>L(G,\vec{x})=\lambda(G).
$$
However, $G'$ does not cover the pair $\{t-1,t\}$, therefore, by Proposition~\ref{prop:maxcolex}, $\lambda([t-1]^{(r)}) \geq \lambda(G')> \lambda(G)$, a contradiction. So 
$$|\mathcal{B}|< \left(\frac{k+1}{k}\right)^{r-2}|\mathcal{A}|\leq \left(1+\frac{r}{k}\right)|\mathcal{A}|.
$$
Thus, 
\begin{align*}
\binom{t}{r}-m&= |[t]^{(r)}\setminus G|\\
&\leq |\mathcal{B}|+|\{C\in [t]^{(r)}: \{t-1,t\}\subseteq C\}\setminus \mathcal{A}| \\
&+ |\{D\in  [t]^{(r)}: |\{t-(k+1)r+1,\dots,t\}\cap D| \geq 3\}|\\
&\leq  \left(1+\frac{r}{k}\right)|\mathcal{A}| + \binom{t-2}{r-2}-|\mathcal{A}|+O(k^3t^{r-3})\\
&\leq \left(1+\frac{r}{k}\right)\binom{t-2}{r-2}+O(k^3t^{r-3})\leq \binom{t-2}{r-2}+\delta_rt^{r-9/4}.
\end{align*}
The last inequality is the explanation for the choice of $k$.
\end{proof}

\section{Concluding remarks}

Closing the remaining gap of the Frankl-F\"uredi Conjecture is a challenging open problem even for $r=3$. As remarked by Talbot in ~\cite{Tb} (with emphasis on $r=3$) the values of $m$ split in two different regimes: $R_1=\{m:\exists t\in \mathbb{N}: \binom{t-1}{r}\leq m \leq \binom{t}{r}-\binom{t-2}{r-2}\}$ and $R_2=\mathbb{N}\setminus R_1$. This split is explained by the fact that in $R_1$ we have  $\lambda(H^{m,r})=\lambda([t-1]^{(r)})$, while in $R_2$ the value $\lambda(H^{m,r})$ `jumps' with every increasing $m$. In the present paper we dealt solely with $R_1$, having confirmed Conjecture~\ref{conj:FF} for all but the $O(t^{r-2})$ largest values of $m$ in each of its intervals (where the interval itself is of length $O(t^{r-1})$). We are wondering, if a version of our argument can be used to deduce $T\leq t$, in which case Theorem~\ref{thm:refine} will further reduce the above range as it already does for $r=3$. With this said, it seems that both completely solving $R_1$  and tackling $R_2$ will require some new ideas.

Talbot's theorem for $r=3$ applies also to small values of $t$. It would be interesting to find an argument that would confirm Conjecture~\ref{conj:FF} in the principal case for \emph{all} $m=\binom{t}{r}$.

With regard to blow-up densities, we have determined, for all $m$ specified by Theorem~\ref{thm:main}, the size of the asymptotically largest blow-up amongst all $r$-graphs of size $m$. We hope that this will find applications in hypergraph Tur\'{a}n problems, as this was the initial motivation behind considering the Lagrangian.
 
\section*{Acknowledgements}

I am grateful to Imre Leader for the helpful feedback.

\end{document}